\documentclass{amsart}

\usepackage{amssymb,amsfonts,latexsym,amsmath,amsthm,wrapfig,graphicx, hyperref, 
}
\input xy
\xyoption{all}

\numberwithin{equation}{section}
\newtheorem{thm}{Theorem}[section]

\newtheorem{lemma}[thm]{Lemma}

\newtheorem{definition}[thm]{Definition}
\newtheorem{conjecture}{Conjecture}
\newtheorem*{question}{Question}

\newcommand{\RR}{\mathbb{R}}
\newcommand{\CC}{\mathbb{C}}

\newcommand{\ZZ}{\mathbb{Z}}

\newcommand{\p}{\partial}

\newcommand{\be}{\begin{equation}}
\newcommand{\ee}{\end{equation}}
\newcommand{\la}{\label}

\newcommand{\rt}{\rm}

\begin{document}

\title[Free boundary problem and the isoperimetric inequality]{A free boundary problem associated with the isoperimetric inequality}

\author[Ar. Abanov]{Artem Abanov}
\email{abanov@tamu.edu}
\address{MS 4242,Texas A\&M University, College Station, TX 77843-4242}

\author[C. B\'{e}n\'{e}teau]{Catherine B\'{e}n\'{e}teau}
\email{cbenetea@usf.edu}
\address{4202 E. Fowler Ave., CMC342, Tampa, FL 33620}

\author[D. Khavinson]{Dmitry Khavinson}
\email{dkhavins@usf.edu}
\address{4202 E. Fowler Ave., CMC342, Tampa, FL 33620}

\author[R. Teodorescu]{Razvan Teodorescu}
\email{razvan@usf.edu}
\address{4202 E. Fowler Ave., CMC342, Tampa, FL 33620}



\begin{abstract}
This paper proves a 30 year old conjecture that disks and annuli are the only domains where analytic content - the uniform distance from $\bar{z}$ to analytic functions - achieves its lower bound.  This problem is closely related to several well-known free boundary problems, in particular, Serrin's problem about laminary flow of incompressible viscous fluid for multiply-connected domains, and Garabedian's problem on the shape of electrified droplets.  Some further ramifications and open questions, including extensions to higher dimensions, are also discussed.
\end{abstract}

\maketitle

\section{Introduction}

This paper solves a long-standing conjecture associated with the isoperimetric inequality, rational approximation, analytic content, and related free boundary problems.
Let $K$ be a compact set in the complex plane. Let $C(K)$ be the algebra of continuous complex-valued functions on $K$ equipped with the standard uniform norm $\|f\|_{\infty} = \sup_{z \in K} |f(z)|$, and let $R(K)$ be the subalgebra of the closure inside $C(K)$ of rational functions with poles off $K$.
The analytic content of $K$ (\cite{GaKh}) is defined by
\begin{definition}\label{ac}
    $$ \lambda(K):= \inf_{\phi \in R(K)} \| \bar{z}-\phi \|_{C(K)}.$$
\end{definition}
Analytic content serves as one of the possible indicators for $K$ to carry a reasonable complex-analytic structure.  For example, it follows at once from the Stone-Weierstrass theorem that $\lambda(K) = 0$ if and only if $R(K)=C(K),$ in other words, the elements of this rather special algebra of rational functions on $K$ are simply generic continuous functions on $K$. The concept of analytic content can be readily extended to deal with other spaces of ``nice" functions, such as the space of harmonic functions (see the concept of harmonic content in \cite{GuKh,Kh87}), or spaces consisting of solutions of more general elliptic equations (\cite{Kh88}).  For more information and for a comprehensive literature review, we refer the reader to the survey \cite{BeKh}.

\begin{figure}[h!!] \begin{center}
\includegraphics*[width=13cm]{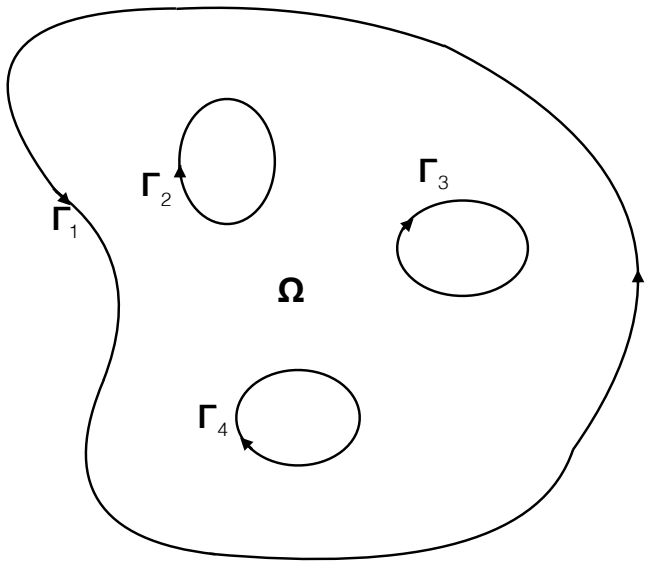}
\caption{The domain $\Omega$ and its boundary components, shown with their orientations relative to $\Omega$ (clockwise for the interior contours, counterclockwise for the exterior one).}
\label{fig1}
\end{center}
\end{figure}

The analytic content of a set enjoys nice estimates in terms of simple geometric characteristics of $K$.  Indeed,
\begin{equation}\label{geo}
    \frac{2Area(K)}{P(K)} \leq \lambda(K) \leq \sqrt{\frac{Area(K)}{\pi}}.
\end{equation}
Here, $P(K)$ is the perimeter of $K$, which is of course finite if for example
$K$ is the closure of a finitely-connected domain with a rectifiable boundary.  Note that otherwise, if $P(K) = \infty$, the left hand side of the inequality is trivial.  We sometimes denote the right hand side by $R_{vol}(K),$ which is the radius of the disk having the same area as $K$. The upper bound of \eqref{geo} was obtained by H. Alexander in \cite{Al} and independently by D. Khavinson in
\cite{Kh,Kh84}.  The lower bound is from \cite{Kh,Kh84}.  Note that \eqref{geo} implies the isoperimetric inequality $P(K)^2 \geq 4 \pi Area(K)$.  The most elementary proof of \eqref{geo} is obtained in \cite{GaKh}, where it is also noted that $\lambda(K)$ attains its upper bound if and only if $K$ is a disk, modulo a ``negligeable" set, that is, a compact set $E$ for which $R(E) = C(E).$  Since in this paper, we will only concern ourselves with domains with reasonable boundaries, we can safely say from now on that $\lambda(K) = R_{vol}(K)$ if and only if $K$ is a disk. The question that was raised in \cite{K87} is then natural:

\begin{question}
    For which sets $K$ is $\lambda(K) = 2 Area(K)/P(K)$?
\end{question}

For the rest of this paper, we will only consider a set $K$ that is the closure of a finitely connected domain $\Omega$ with a real analytic boundary $\Gamma = \cup_{k=1}^n \Gamma_k$, and we write $\lambda:=\lambda(\Omega) : = \lambda(K)$.  The following conjecture goes back to \cite{Kh84,K87}:

\begin{conjecture}\la{Main}
    We have $\lambda(\Omega) = 2 Area(\Omega)/P(\Gamma) $ if and only if $\Omega$ is a disk or an annulus.
\end{conjecture}
For simply-connected $\Omega$, the extremal domains are known to be only disks.  This was proved in \cite{K87}.  Conjecture \ref{Main} can be expressed in several equivalent forms, as follows.

\begin{thm}[\cite{K87,GuKh}]\la{DK}
Let $\Omega$ and $\Gamma$ be as above.  The following are
equivalent:

(i) $\lambda=\frac{2Area(\Omega)}{P(\Gamma)};$

(ii) There is $\varphi$ analytic in $\overline{\Omega}$ such that
$\bar{z}(s)-i\lambda\frac{d\bar{z}}{ds}=\varphi(z(s))$ on $\Gamma,$ where
$s$ is the arc-length parameter;

(iii) The following quadrature identity
\begin{equation} \la{quadrature}
    \frac{1}{Area(\Omega)}\int_{\Omega} f dA = \frac{1}{P(\Gamma)}\int_{\Gamma} f ds
\end{equation}
holds for all bounded analytic functions $f$ in $\Omega,$ where $dA$ denotes area measure in $\CC$.

(iv) There exist constants $c_1, c_2, \ldots, c_n \in \RR$ such that the overdetermined boundary value problem
\begin{equation}\la{bvp}
    \begin{cases}
        \Delta u =1 \,\, in \,\, \Omega \\
       \frac{\partial u}{\partial n}= \frac{Area(\Omega)}{P(\Gamma)} \,\, on \,\, \Gamma \\
       u\mid_{\Gamma_k} = c_k \,\, k = 1, 2, \ldots, n
    \end{cases}
\end{equation}
has a smooth solution in $\Omega$.  (Here, $\frac{\partial }{\partial n}$ denotes the exterior normal derivative on $\Gamma$.)
\end{thm}

In view of the fact that \eqref{ac} yielded a new proof of the isoperimetric inequality and because of the quadrature identity \eqref{quadrature}, Conjecture \ref{Main} and its ramifications have attracted the attention of a number of mathematicians (see \cite{BeKh,KSV} and references therein). Condition \textit{(iv)} in the above theorem allows one to connect problems involving analytic content to a class of free boundary problems known as J. Serrin's problem (see, e.g., \cite{Se,We,BeKh,GuKh,KSV}).

The rest of the paper is organized as follows.  In Section \ref{phys}, we discuss four physical free boundary problems associated with Conjecture \ref{Main}.  We then turn to a proof of Conjecture \ref{Main}:  in Section \ref{reduction}, we reduce the problem to domains of connectivity at most 2, and in Section \ref{DP}, we prove the conjecture for doubly-connected domains.  We close with some final remarks and open questions.

\textbf{Acknowledgements.} The work on this paper began during the special semester on ``Complex Analysis and Integrable Systems" at the Mittag-Leffler Institute.  The authors would like to thank the Institute and the organizers for their support. Ar. A. is also very grateful for the warm hospitality of the INSPIRE
group in Johannes Gutenberg-Universit{\"a}t, Mainz, Germany and its
supporting staff. D. K. was partially supported by NSF grant DMS-0855597.

\section{{{Four}} physical problems}\la{phys}
In this section, we consider {\rt{four}} physical problems related to Conjecture 1, {\rt{two}} concerning Newtonian fluid flow, and the other {\rt{two concerning classical and quantum electrically-charged liquids in two dimensions}}.

\subsection{Serrin's Problem}

J. Serrin's problem concerns a laminary flow of a viscous, Newtonian fluid in a pipe with cross-section $\Omega$. (For an expository description of this problem, see, for example, \cite{BeKh}.)  Let the $z$ axis be directed along the length of the pipe while $(x,y)$ represent the coordinates in the pipe's cross-section. {\rt{Neglecting gravity, in the limit of viscous flows with low Reynolds numbers, the Navier-Stokes equations reduce to the Stokes equations, pressure can be taken to be a linear function of $z$, and velocity has only one non-vanishing component, $\vec{v} =(0, 0, u)$.}} As {\rt{the flow is laminar}}, the velocity $u$ of the fluid particle is the same along every streamline, that is, $ u = u(x,y)$ does not depend on $z$.  Since the rate of change of pressure $p$ along the pipe is constant $\frac{dp}{dz} = C,$ the {\rt{Stokes}} equations reduced to this situation yield that $-\Delta u = \nu^{-1}\frac{dp}{dz} = C/\nu$, {\rt{where $\nu$ is the coefficient of dynamic viscosity}}.  The tangential stress of the viscous fluid on the pipe walls is proportional to the normal derivative of the velocity.  One can imagine that {\rt{along}} the pipe walls (the boundary components of $\Gamma$), {\rt{the fluid is}}  either at rest ($u = 0$ on $\Gamma$, {\rt{or no-slip condition}}) or moving with perhaps different velocities ($u = c_k$ on $\Gamma_k$).  For the simply-connected case ($n = 1$, $c_1 = 0$), Serrin (\cite{Se}) proved that if the tangential stress on the pipe is constant, then $\Omega$ is a disk.  This, as noted in \cite{K87}, proves Conjecture \ref{Main} under the additional assumption that $\Omega$ is simply-connected: in that case, $\Omega$ must be a disk. An independent proof of Serrin's theorem for $\RR^2$ that is based solely on the use of (ii) in Theorem \ref{DK} is due to Gustafsson and can be found in \cite{K87}, or in \cite{BeKh}.  Various partial cases of Conjecture \ref{Main} in the form of \textit{(iv)} in Theorem \ref{DK}, with assumptions on the constants $c_k, \, k = 1, \ldots, n,$ were treated by many authors (see the references in \cite{BeKh,KSV}). The physical requirement corresponding to the third equality in \eqref{bvp} is that the ``drag" force on the pipe is constant along the perimeter.  Most extra assumptions are reduced to having $c_1$ be the largest of all the constants so that the Serrin-Alexandrov moving plane method can be applied, yielding spherical symmetry of $\Omega$.  This forces $\Omega$ to be a spherical shell.  (See \cite{GuKh, BeKh} for multi-dimensional analogues of conditions \textit{(i)} through \textit{(iii)} of Theorem \ref{DK} and relevant discussions.)  Yet, without additional assumptions on the boundary values $c_k, \, k = 1, \ldots, n,$ of $u$ on $\Gamma$ (as in (iv) of Theorem \ref{DK}), Conjecture \ref{Main} remained open.

\subsection{The shape of an electrified droplet}

If we consider a droplet of perfectly conducting fluid in the plane, with given electrostatic potential $\Phi$, there are three forces acting on the free boundary of the droplet:  the electrostatic force $\vec{F}_{el}$ and the force due to pressure $\vec{F}_{pr}$, both trying to tear the droplet apart, and the force due to surface tension, $\vec{F}_{sf}$, trying to keep the droplet together.  Let us sketch a derivation of the equation for the free boundary $\Gamma$ of the droplet in equilibrium.  See \cite{Ga,BeKh,KSV} for more details and references.

The {\rt{equilibrium}} electrostatic force acting on a piece of the boundary of $\Gamma$ {\rt{of infinitesimal length $ds$}} is $\vec{F}_{el} \sim |\vec{E}|^2 \vec{n} ds,$ where $s$ is arc-length, $\vec{n}$ is the outward unit normal vector, and $\vec{E} = \nabla U$ is the electrostatic field {\rt{(since the linear charge density in equilibrium is proportional to the normal component of the electrostatic field, $\frac{d \rho}{ds} \sim \vec{n} \cdot \vec{E} = |\vec{E}|$).}}  Here the harmonic function $U$ is the {\rt{ electrostatic}} potential.  Set $\Phi = U + i V$ to be the analytic potential corresponding to $U$.  Then $$\vec{E} = 2 \frac{\partial}{\partial \bar{z}} \left( \frac{\Phi + \bar{\Phi}}{2} \right) = \overline{\frac{\partial \Phi}{\partial z}}.$$  Here,
$\frac{\partial}{\partial \bar{z}} = \frac{1}{2} \left( \frac{\partial}{\partial x} + i \frac{\partial}{\partial y}\right)$ so that $\nabla U = 2 \frac{\partial U}{\partial \bar{z}}$. Assuming $\Gamma$ to be real-analytic, it can be parametrized by its Schwarz function $S$, analytic in a  neighborhood of $\Gamma$ so that $\Gamma = \{ z : \bar{z} = S(z) \}$ (see \cite{Da,Sh}).  Then, since
$ 1 = \frac{d\bar{z}}{ds} \frac{dz}{ds} = S'(z) \left( \frac{dz}{ds}\right)^2,$  the normal $\vec{n} = - i \frac{dz}{|dz|} = - i /\sqrt{S'(z)}$.  Thus,
$$ \vec{F}_{el} \sim \frac{-i}{\sqrt{S'(z)}} |\partial \Phi|^2 ds.$$

Moreover, the electric field $$\vec{E} = \overline{\partial \Phi} = |\partial \Phi| \left( - \frac{i}{\sqrt{S'}} \right).$$ Hence,
$$|\partial \Phi| = i \sqrt{S'} \, \overline{\partial \Phi} = - i /\sqrt{S'} \, \partial \Phi$$ as $|S'| = 1$ on $\Gamma$.  Thus
\begin{equation}\la{el}
  \vec{F}_{el} \sim \frac{-i}{\sqrt{S'}}  \left( - \frac{1}{S'} \right) \left( \partial \Phi \right)^2 ds = \frac{i}{(S')^{3/2}} \left( \partial \Phi \right)^2 ds.
\end{equation}

Now, the surface tension is proportional to the curvature, that is, $\vec{F_{sf}} \sim \frac{d \tau}{ds} ds$, where
$\tau = \frac{dz}{|dz|} = \frac{dz}{ds} = \frac{1}{\sqrt{S'}}$ is the unit tangent vector to $\Gamma$. Since $\frac{d}{ds} = \frac{1}{\sqrt{S'}}\frac{d}{dz}$, we arrive at
\begin{equation}\la{sf}
\vec{F}_{sf} \sim \frac{ds}{\sqrt{S'}} \frac{d}{dz} \left( \frac{1}{\sqrt{S'}}\right).
\end{equation}
Now the force due to pressure, we simplify to be
\begin{equation}\la{pr}
\vec{F_{pr}} \sim \vec{n} ds = - \frac{i}{\sqrt{S'}} ds.
\end{equation}
If the droplet is in equilibrium, the sum of the forces \eqref{el},\eqref{sf}, \eqref{pr} must be 0, and we obtain
\begin{equation*}
    c_1 \left(\frac{\partial \Phi}{dz} \right)^2 \frac{i}{(S')^{3/2}}  ds + c_2 \frac{1}{\sqrt{S'}} \frac{d}{dz} \left( \frac{1}{\sqrt{S'}}\right) ds + c_3  \frac{- i}{\sqrt{S'}} ds = 0,
\end{equation*}
where $c_1,c_2,$ and $c_3$ are real constants. Equivalently,
$$ c_1 \left( \frac{\partial \Phi}{dz} \right)^2   -i  c_2 S' \frac{d}{dz} \left( \frac{1}{\sqrt{S'}}\right)  - c_3  S'  = 0.$$ Noticing that
$S' \frac{d}{dz} \left( \frac{1}{\sqrt{S'}}\right) = - \frac{1}{2 \sqrt{S'}} \frac{dS'}{dz}$ while
$ \frac{d \sqrt{S'}}{dz}  = \frac{1}{2 \sqrt{S'}} \frac{dS'}{dz},$ we obtain
\begin{equation}\la{forcesum}
    c_1 \left( \frac{\partial \Phi}{dz} \right)^2   +i  c_2 \frac{d}{dz} \left( \sqrt{S'}\right)  - c_3  S'  = 0.
\end{equation}
Now define $F(z) = c_1 \int \left( \frac{\partial \Phi}{dz} \right)^2 dz.$ Then \eqref{forcesum} becomes, after integration,
\begin{equation}\la{intforcesum}
    F(z) + i c_2 \sqrt{S'(z)} - c_3 S(z) = 0.
\end{equation}
Dividing by $c_3$, and renaming $F(z) / c_3 = \varphi (z)$ and $c_2/c_3 = \lambda,$ and using the fact that $S(z) = \bar{z}$ on $\Gamma$ and $\sqrt{S'(z)} = \frac{d \bar{z}}{ds}$, we arrive at
\begin{equation}\la{DK2}
     \bar{z}(s) - i \lambda \frac{d \bar{z}}{ds} = \varphi(z),
\end{equation}
which is precisely the equation $(ii)$ given in Theorem \ref{DK}.

Several remarks are in order.

(i) If the potential $U$ has a point charge at $z_0$, then $\varphi(z) \sim \frac{const}{z - z_0}$ near $z_0$, that is, $\varphi$ has a pole.

(ii) Usually, for a physical droplet, the fluid is assumed to be incompressible.  Then either the area is assumed to be fixed, or the area, the pressure, and the temperature are connected by the ``equation of state".  In particular, for an incompressible fluid, the pressure has to be adjusted each time the area is fixed.  If we amend the problem with this requirement, the physical picture is the following.  Consider a plane with a system of charges on it.  We throw a droplet of fluid onto the plane and see where it will come to rest and what shape it will have.  For example, if there is only one charge, this charge will induce a dipole moment on the droplet, and the dipole will move to ``swallow" the charge.  Then, there will be no charge outside, and the charge inside will redistribute itself over the surface, while at $\infty$ we still have $\varphi \sim \frac{c}{z}$. Thus, as in example (i) with $z_0=0$, \eqref{DK2} would become, for some constant $c \in \RR$,
\begin{equation*}
    \bar{z}(s) - i \lambda \frac{d \bar{z}}{ds} = \frac{c}{z},
\end{equation*}
or
\begin{equation}\la{dipole}
    S(z) - i \lambda \sqrt{S'(z)} = \frac{c}{z}.
\end{equation}
Denoting $u(z) = \sqrt{S'(z)}$ and differentiating with respect to $z$, we reduce \eqref{dipole} to
$$ u^2 - i \lambda u ' = - \frac{c}{z^2},$$
the Ricatti equation, and the unique solution $u = const/z$ is easily found.  This implies that $ S' = const/z$ and $\Gamma $ is a circle centered at the origin.  Note that a ``physical" solution yields the same result without any calculation, merely by noticing that $const/z$ is radially symmetric ($U = \log |z|$), and therefore the problem must have a radially symmetric solution, hence, a circle. The Ricatti equation plays a crucial role in the proof of Conjecture \ref{Main} in subsequent sections, see also \cite{K87}.

(iii) Let us look again at \eqref{DK2}, where $\varphi(z) = const \int \left( \frac{\partial \Phi}{dz} \right)^2 dz,$ and $\Phi = U + i V$ is the analytic potential.  Then  \eqref{DK2} enforces an extra condition on the problem, namely that $\sqrt{\varphi'}$ is a single-valued function.  In general, if $U(z) = \int_{\CC} \log|z - \zeta| d \mu(\zeta)$ is an arbitrary potential of a charge distribution $\mu$, then
$$ \sqrt{\varphi'} = const \frac{\partial \Phi}{\partial z} = \int \frac{d \mu(\zeta)}{\zeta - z},$$
a single-valued function.  We will call the solution to the problem \eqref{DK2} with $\sqrt{\varphi'}$ single-valued a physical droplet versus a mathematical droplet if not (see the discussion in \cite{KSV}).

(iv) Note that the free boundary problem \eqref{DK2} is extremely restrictive.   As was already noted in \cite{K87}, if the free boundary $\Gamma$ contains a circular arc, then the extremal domain must either be a disk (of radius $\lambda$) or an annulus.  Indeed, if say $\Gamma$ contains a circular arc centered at the origin of radius $R$, then \eqref{DK2} implies that either $\varphi \equiv 0$ if $\lambda = R$ and $\Gamma = \{ z: |z| = R\},$ or $\varphi = const/z$, so that every connected component of $\Gamma$ is a circle centered at the origin, and therefore $\Omega$ must be an annulus.

(v) Finally, we mention that a slightly more general free boundary problem
\begin{equation}\la{fbp}
    p \bar{z} - i t \frac{d \bar{z}}{ds} = F(z),
\end{equation}
where $F$ is a given analytic or meromorphic function and $t$ is a real parameter, was discussed in \cite{KSV} in detail.  In particular, choosing $p = 0$ and $F$ analytic in $\CC - \Omega$  (the complement of the droplet) with a simple pole at infinity gives rise to an interesting family of non-circular algebraic droplets depending on the value of the parameter $t$ (see \cite{KSV} for details).  The easier version of the latter problem with $F$ analytic in $\overline{\CC} - \Omega$ (including infinity) was considered in \cite{EKS} in connection with the study of the first eigenvalue of the spectrum of the single layer potential.

\subsection{Incompressible flows in 2D and generalized Rankine vortices}

Incompressible flow dynamics in two dimensions with non-vanishing vorticity have a distinguished history \cite{G_Stokes,Helmholtz,Lamb}, with some important open problems relevant to the field-theoretic extension applicable to Quantum Hall systems and other 2D strongly-interacting quantum electronic systems \cite{Wiegmann}. We briefly review here the connections between this class of problems and the isoperimetric inequality, referring the reader to \cite{RT2014} for a more in-depth  quantum-field theoretic discussion.

\subsubsection{Classical 2D incompressible vortex flows} \la{above}

In 2D classical incompressible flows, the problem of equilibrium distribution of vorticity is particularly relevant because of its connection to the onset of turbulence (hence, to regularity of solutions for the Navier-Stokes equations).  Incompressible 2D velocity fields $\vec{v}$ can be expressed in complex notation as $\vec{v} = v_x + i v_y = 2 i \bar{\p} \psi$, where the stream function $\psi(x, y)$  is real-valued, solving the Poisson equation $\nabla^2 \psi = \omega(x, y)$, and $\omega(x, y) = \p_x v_y - \p_y v_x$ is the 2D vorticity field of the flow. Flow incompressibility follows directly from
\be
i \nabla^2 \psi = 2 \p (v_x + i v_y) = \vec{\nabla}\cdot \vec{v} + i \omega \in i \mathbb{R}.
\ee
This formulation is useful because it allows to express time-independent solutions to incompressible flows in 2D entirely via the stream function $\psi(x, y)$. For example, irrotational,  incompressible flows are equivalent to boundary-value problems for the Laplace operator, since $\psi$ is a harmonic function in the domain of irrotational flow.  In general, the problem requires finding the equilibrium distribution of the vorticity field, decomposable as the sum of an absolutely-continuous part and a singular part.

In this formulation, the problem discussed in this paper requires finding a bounded domain $\Omega$ of connectivity $n \ge 1$, and a stream function $\psi(x, y)$
\be
\psi(x, y) = |z|^2 - 2 \Re \int \varphi(z) dz,
\ee
such that $\omega(x, y) = \nabla^2 \psi = $ constant in $\Omega$, and the velocity field $\vec{v} = 2i\bar{\p} \psi = 2i(z-\overline{\varphi}(z)) =  2\lambda \tau$ on $\p \Omega$, that is the boundary of $\Omega$ consists of streamlines with constant (tangent) velocity $|\vec{v}| = \lambda$, where $\tau$ represents the unit tangent vector, $\lambda$ is the analytic content, and $\varphi$ is the best approximation to $\bar{z}$, as in Theorem~\ref{DK}.  The singularity set where $\varphi$ is not analytic (inside the complement of $\overline{\Omega}$) will correspond to the singular distribution of vorticity, while inside $\Omega$ vorticity is constant, $\omega(x, y) = 4$. Applying Green's theorem to the vector field $\vec{v}$ on $\Omega$ leads to the expected identity
\be \la{ver}
4 {\rm{Area}}(\Omega) =  \int_{\Omega} \omega dx dy = \oint_{\p \Omega} \vec{v} \cdot \tau ds = 2 \lambda P.
\ee

It is instructive to notice that the simply-connected case ($n =1$) was shown long ago to correspond to a disk domain, and the associated vorticity distribution is known as the Rankine vortex \cite{Rankine}.

\subsubsection{Chiral fields in conformal theories with several boundary components} \la{quantum}

Not surprisingly, the 2D equilibrium distribution vorticity problem described in \S\ref{above} has a magneto-static counterpart, in which we require finding a domain $\Omega$ such that the total magnetic field is oriented along the direction perpendicular to the $(x,y)$ plane, and whose intensity $B(x, y)$ is constant in $\Omega$ (more precisely, we can take $B = 4$ in $\Omega$ to make explicit the analogy with the vorticity field from \S\ref{above}), as well as having singularities in the complement $\overline{\Omega}^c$, corresponding to infinitely-narrow magnetic flux tubes, and given by the singularity set of $\varphi(z)$.

The time-independent vector potential $\vec{A}(x, y) = (A_x, A_y, 0)$ is fixed by the gauge condition $\vec{\nabla}\cdot \vec{A} = 0$, so that
\be
2{\p} (A_x + i A_y) = \vec{\nabla}\cdot \vec{A} + i B(x, y).
\ee
Just as in \S\ref{above}, the choice $A_x + iA_y = 2i(z - \bar{\varphi}(z))$ solves all the constraints, with the additional requirement that $\vec{A} = \lambda \tau$ on $\p \Omega$, where again $\lambda, \tau$ represent the analytic content, and the tangent unit vector, respectively. This means that the boundary components of $\Omega$ can be identified with   closed loops of electrical current, and the vector potential has constant magnitude on $\p \Omega, |\vec{A}| = \lambda$. Green's theorem for the field $\vec{A}$ provides again the expected identity \eqref{ver}, in fact the condition  (i) of Theorem~\ref{DK}.

The classical vortex flow problem has a quantum correspondent \cite{RT2014}, related to open problems in conformal field theory (CFT). It is a boundary CFT problem requiring finding    a domain $\Omega$ (as indicated above), with (holomorphic) energy-momentum tensor density $T_{zz} = \varphi'(z)$. Since $\varphi(z)$ is analytic in $\Omega$,  $T_{zz} dz^2 = \varphi'(z) dz^2$ must be a quadratic differential in $\Omega$ (as will be indeed discussed in the next section). The chiral fields $v_k(z) = exp[-i\lambda^{-1} \int^z u_k(\zeta) d\zeta],  z, \zeta \in \Gamma_k$, analytically continued into $\Omega$, satisfy the projective connection \cite{Frenkel} null condition
\be
\left [ \frac{d^2 \,\,}{d z^2} + \frac{\varphi'(z)}{\lambda^2} \right ] v_k = 0, \quad z \in \Omega,
\ee
while the gauge fields $u_k(z)$ (which reduce to the vector potential fields $\vec{A}$ on $\p \Omega$) are consistently related to the energy-momentum tensor $T_{zz}$ via the covariant derivative (or momentum) constraint
\be
\nabla_{u}(u) = \left ( \frac{d}{dz} - \frac{i}{\lambda} u \right ) u = \frac{i}{\lambda} \varphi'(z), \quad z \in \overline{\Omega},
\ee
which is equivalent to differentiating (ii) in Theorem~\ref{DK} with respect to $z$.

\section{Reduction to the doubly-connected case} \la{reduction}

Let us now turn to a proof of Conjecture 1.  Assume $\Omega$ is a finitely-connected extremal domain, that is, a domain such that
$\lambda(\Omega) = \frac{2 \, Area(\Omega)}{P(\p \Omega)}$,
with boundary components $\Gamma_k, k = 1, 2, \ldots, n$ with $n \geq 2$. In this section, we will show that $\Omega$ must be doubly-connected ($n = 2$).

Denote by $\{ \Omega_k \}_{k=1}^n$ the domains defined by
$\Omega_k \cap \Omega = \emptyset,  \,\, \partial \Omega_k = \Gamma_k,$
and choose $\Omega_1$ for the one which is unbounded.  Recall that $\tau = \frac{dz}{ds}$ is the unit tangent vector at $z \in \Gamma$, and
define $\kappa$ to be the signed curvature at $z \in \Gamma,$ that is
$$ \kappa = -i \tau \cdot \frac{d\bar{\tau}}{ds} = - i \frac{\frac{d^2\bar{z}}{ds^2}}{\frac{d \bar{z}}{ds}}.$$ Notice that $\kappa$ is real. We then have the following.
\begin{thm} \la{function}
Let $\Omega$ be an extremal domain, let $\varphi$ be the best approximation of $\bar{z}$, and let $\lambda$ be the analytic content of $\Omega$.  Then  $\varphi'(z)dz^2$ is a quadratic differential that is real-valued on $\p \Omega$, and
\be
\varphi'(z) dz^2 = (1+ \lambda \kappa) ds^2
\ee
along each component $\Gamma_k$ of $\p \Omega.$
Moreover, on every component $\Gamma_k$ of $\p \Omega$,
$
\oint_{\Gamma_k} (1 +  \lambda \kappa) ds > 0.
$
\end{thm}
\begin{proof}
By Theorem \ref{DK}, $\varphi$ satisfies
\be
\bar{z}(s)-i\lambda\frac{d\bar{z}}{ds}=\varphi(z(s))
\ee
on $\Gamma,$ where
$s$ is the arc-length parameter.  Differentiating with respect to arc-length gives
\be
\frac{d\bar{z}}{ds} - i \lambda \frac{d^2\bar{z}}{ds^2} = \varphi '(z) \frac{dz}{ds}.
\ee
Dividing by $\frac{d\bar{z}}{ds}$, using the fact that $s$ is arc-length, and by definition of $\kappa$, we arrive at
\be
1 + \lambda \kappa = \varphi '(z) \left( \frac{dz}{ds}\right)^2,
\ee
or, equivalently,
\be \la{l0}
{\varphi'(z)}dz^2 = (1+\lambda \kappa) ds^2,  \quad z \in \Gamma_k.
\ee
Since $\varphi'$ is analytic and since the right hand side of \eqref{l0} is real, ${\varphi'(z)}dz^2$ is a quadratic differential that is real-valued on
$\p \Omega$.

Now notice that for any contour $\Gamma_{k}$,
\begin{eqnarray*}
\int_{\Gamma_k} \kappa ds & = & -i \int_{\Gamma_k} \frac{\frac{d^2\bar{z}}{ds^2}}{\frac{d \bar{z}}{ds}} ds \\
& = & -i \Delta_{\Gamma_k} \left( \log \left( \frac{d \bar{z}}{ds} \right) \right)\\
& = &  \Delta_{\Gamma_k} \left( \arg \frac{d \bar{z}}{ds} \right).
\end{eqnarray*}
For $k = 1$, this value is equal to $-2 \pi$, while for $k \geq 2,$ we get $ 2 \pi$.
Therefore for any interior contour $\Gamma_{k}, \, k \ge 2$,
we obtain
\be
\int_{\Gamma_k} (1+ \lambda \kappa)ds = L_k + 2\pi\lambda > 0.
\ee
On $\Gamma_1$, we have
\be
\int_{\Gamma_1} (1+ \lambda \kappa)ds = L_1 - 2\pi\lambda = L_1 - \frac{4\pi A}{P},
\ee
with $A = \mbox{Area}(\Omega)$ and $P = L_1 + \sum_{k \ge 2}L_k$ its perimeter. Using $P \ge L_1$, we see that
\be
L_1 - \frac{4\pi A}{P} \ge L_1 - \frac{4\pi\mbox{Area}(\Omega)}{L_1} \ge \frac{4\pi}{L_1}\left [ \mbox{Area}(\Omega_1^c) -
\mbox{Area}(\Omega) \right ] > 0,
\ee
where we have used the isoperimetric inequality for the complement of $\Omega_1$, $\Omega_1^c$, and the fact that
$\Omega \subseteq \Omega_1^c$.
\end{proof}

Now recall that as discussed in the introduction, if $S_k(z)$ are the Schwarz functions for $\Gamma_k$, that is, $S_k$ is analytic in a neighborhood
of $\Gamma_k$ and satisfies $S_k(z) = \bar z$ on $\Gamma_k$, then for $u_k(z) = \sqrt{S'_k(z)}$, the functions $u_k$ satisfy the Ricatti equation
\be \la{riccati}
u_k^2 + i\alpha \lambda u_k' = \varphi'(z), \, z \in \Omega,
\ee
where $\alpha = 1$ for $k \geq 2$ and $\alpha = -1$ for $k = 1$.  By a standard reduction, the functions
\be \la{linsol}
v_k(z) := \exp \left [-\frac{i \alpha}{\lambda}  \int^z u_k(\zeta) d\zeta \right ], \quad k = 1, 2, \ldots, n,
\ee
solve the linear second-order differential equation associated with \eqref{riccati}
\be \la{ODE}
v'' = -\frac{\varphi'}{\lambda^2} v.
\ee

\begin{definition} \la{nw1}
Let $\Sigma^{\pm}$ be the union of Stokes and anti-Stokes graphs of \eqref{ODE} in $\Omega$ \cite[Lemma~9.2-1]{Olver}, i.e. the union of arcs  $\{ \gamma^{\pm}_j\}$ satisfying
$$
\Im \int_{z_0}^z \sqrt{\varphi'(\zeta)} d\zeta = 0, \quad \zeta \in \gamma^{+}_j \subset \Sigma^{+}, \quad
\Re \int_{z_0}^z \sqrt{\varphi'(\zeta)} d\zeta = 0, \quad \zeta \in \gamma^{-}_j \subset \Sigma^{-},
$$
where $z_0$ is any zero of $\varphi'(z)$ in $\overline{\Omega}$. 
\end{definition}

It is known \cite{Olver,D} that if $\varphi'(z)$ is analytic in $\Omega$, then $\Sigma^{+}, \Sigma^{-}$ have the same number of arcs $\gamma^{\pm}_j$, they intersect only at zeros of $\varphi'(z)$, and each arc $\gamma^{\pm}_j$ is analytic, with one endpoint being a zero of $\varphi'$, and the other being either another zero, or a point on $\p \Omega$ (or possibly, both). Moreover, at a zero $z_0 \in \Omega$ of $\varphi'$ of order $m \ge 1$, there are exactly $m+2$ arcs from $\Sigma^{+}$ with local angle between adjacent arcs equal to $2\pi/(m+2)$, and another $m+2$ arcs from $\Sigma^{-}$, each of them bisecting the angle between two consecutive arcs of $\Sigma^{+}$.

Let $z_0 \in \overline{\Omega}$ be a zero of order $m$ of $\varphi'$. By elementary calculations, it is easy to describe the local power series expansion of $v$ about $z_0$,
but the local solution is not convenient to use when exploring global properties of solutions such as $|v_k(z)|_{\Gamma_k} = $ constant, satisfied by \eqref{linsol}. Instead, we will examine the asymptotic series representations, valid outside a small neighborhood of $z_0$. Defining the local coordinates $\zeta = \epsilon(z-z_0)$, with $\epsilon$ a scale parameter, arbitrarily small but strictly positive, then c.f. \cite[Ch. 6]{LG}, \cite[Ch. 3]{Olver}, \cite{FF}, the general solution for Eq.~\eqref{ODE} admits the asymptotic series representation known as Liouville-Green (LG) in applied mathematics and Jeffreys-Wentzell-Kramers-Brillouin (JWKB) in theoretical physics
\be \la{stokes}
v(\zeta,\epsilon) = \frac{\sqrt{\lambda}}{(\varphi')^{1/4}} \left [ C_1 e^{\frac{i}{\lambda \epsilon}\int_0^\zeta  \sqrt{\varphi'} d\xi} + C_2 e^{-\frac{i}{\lambda \epsilon}\int_0^\zeta  \sqrt{\varphi'} d\xi} \right ][1 + o(\epsilon)],
\ee
where $C_{1, 2}$ are constants, and $\zeta$ belongs to a domain $D$ having 0 as boundary point. In particular, for $z \in \Sigma^{+}$, the domain of validity includes a wedge domain of angle $2\pi/(m+2)$, with $\Sigma^{+}$ bisecting the angle.  The solution is approximated by the asymptotic expansion in the sense of the Borel-Ritt theorem \cite[\S~3.3]{Olver}, i.e., the R.H.S. of \eqref{stokes} is a function of $\zeta$, smooth in both  $\zeta$ and $\epsilon$, and
\be \la{conv}
\lim_{\epsilon \to 0}\frac{1}{\epsilon} \left [ v(\zeta, \epsilon) - \frac{\sqrt{\lambda}}{(\varphi')^{1/4}} \left (C_1 e^{\frac{i}{\lambda \epsilon}\int_0^\zeta  \sqrt{\varphi'} d\xi} + C_2 e^{-\frac{i}{\lambda \epsilon}\int_0^\zeta  \sqrt{\varphi'} d\xi} \right ) \right ] = 0, \, \zeta \in D.
\ee
Let us use this asymptotic expansion to examine the potential zeros of $\varphi'$.

\begin{lemma} \la{n2}
The function $\varphi'$ cannot vanish at any point on $\p \Omega$, so the quadratic differential $\varphi'(z) dz^2$ is strictly positive-definite on $\p \Omega$.
\end{lemma}
\begin{proof}
Assume that $\varphi'(z_0) = 0, z_0 \in \Gamma_k \subset \p \Omega$. Then by Theorem~\ref{function}, $\Gamma_k \subset \Sigma^{+}\cup \Sigma^{-}$. The two arcs $\gamma_{1,2}(z_0)$ of $\Gamma_k$ meeting at $z_0$ are elements either of $\Sigma^{+}$ or of $\Sigma^{-}$. However, at least  one such arc must belong to $\Sigma^{+}$, because otherwise $\Gamma_k \subset \Sigma^{-}$, which implies that $\varphi' dz^2$ is negative-definite on $\Gamma_k$, so according to \eqref{l0}, $1 + \alpha \lambda \kappa \le 0$ everywhere on $\Gamma_k$, which contradicts Theorem~\ref{function}.

Take now $z$ on the arc belonging to $\Sigma^{+} \cap \Gamma_k$.   According to the LG formula \eqref{stokes}, the solution \eqref{linsol} has the asymptotic expansion
\be
v_k(z_0 + \epsilon \zeta) =
\frac{\sqrt{\lambda}}{(\varphi')^{1/4}} \left [ C_1 e^{\frac{i}{\lambda \epsilon}\int_0^\zeta  \sqrt{\varphi'} d\xi} + C_2 e^{-\frac{i}{\lambda \epsilon}\int_0^\zeta  \sqrt{\varphi'} d\xi} \right ][1 + o(\epsilon)]
\ee
with $C_{1,2}$ constants. Denote by $\gamma = \Sigma^{+} \cap \Gamma_k \cap D$, and notice that along $\gamma$, condition \eqref{conv} and \eqref{linsol} give
\be \la{kn}
\lim_{\epsilon \to 0}\frac{1}{\epsilon} \left [ e^{-i\frac{\alpha}{\lambda \epsilon}s(\zeta)} - \frac{\sqrt{\lambda}}{(\varphi')^{1/4}} \left (C_1 e^{\frac{i}{\lambda \epsilon}\int_0^\zeta  \sqrt{\varphi'} d\xi} + C_2 e^{-\frac{i}{\lambda \epsilon}\int_0^\zeta  \sqrt{\varphi'} d\xi} \right ) \right ] = 0.
\ee
Take $z \in \gamma$ so that the arclength along $\gamma$ from $z_0$ to $z$, is $s > 0$. Let $\omega(s) := \int_{z_0}^z  \sqrt{\varphi'} d\xi$ and note that $\omega(s) > 0$ by the choice of $z$.  Also, let $K_{1, 2}(z) :=
\frac{\sqrt{\lambda}}{(\varphi'(z))^{1/4}} C_{1,2}$ and consider first the case of an interior boundary component $\Gamma_k$, i.e., $\alpha = 1$. Condition \eqref{kn} implies then, that for a fixed $z$,
\be \la{yo}
\lim_{\epsilon \to 0}\left |1  - K_1 e^{\frac{i}{\lambda \epsilon}(s+\omega(s))} - K_2 e^{\frac{i}{\lambda \epsilon}(s-\omega(s))}\right | = 0.
\ee
Taking now the sequence $\epsilon_n \equiv \frac{s + \omega(s)}{2 \pi \lambda n}, n \in \mathbb{N}$, we obtain
\be
\lim_{n \to \infty}\left |1  - K_1 - K_2 q^n\right | = 0, \quad
q = e^{2\pi i\frac{s-\omega(s)}{s + \omega(s)}} \in \mathbb{T}.
\ee
This is possible either if $K_1 = 1, K_2 = 0$ for arbitrary $q$, or if $K_1 + K_2 = 1, q=1$.
Therefore,  $K_1+K_2 = 1$, so $|\varphi'(z)| $ is a constant (and hence $0$, since $\varphi'(z_0) = 0$) along the arc $\gamma$. But then $\varphi'$ is identically $0$, which cannot happen unless $\Omega$ is  a disc, which is a contradiction.

For the case of the exterior boundary $\alpha = -1$, we exhange $K_1$ and $K_2$ in \eqref{yo} and the argument follows identically.
\end{proof}

\begin{thm}
The domain $\Omega$ is a {\emph{maximal domain}} in the sense of \cite{Jenkins}, so its connectivity (and the total number of boundary components of $\p \Omega$) is 1 or 2.
\end{thm}
\begin{proof}
Note that there cannot be any open arcs of  $\Sigma^{+}$ in $\overline{\Omega}$, because by the properties of $\Sigma^{+}$ discussed earlier, such an arc would have to end
at a zero of $\varphi'$ on $\p \Omega$, which is prohibited by Lemma~\ref{n2}. Moreover, the Stokes graph is connected, and it contains $\p \Omega$. Therefore, any trajectory (in the sense of \cite{Jenkins}) of the quadratic differential $\varphi'(z) dz^2$ that includes $\Sigma^{+}$ must be
 a closed curve in $\p \Omega$, so $\Omega$ is a {\emph{maximal domain}}. Then from \cite[Theorem 1]{Jenkins}, the connectivity of $\Omega$ cannot exceed 2.
\end{proof}

\section{Solution for the doubly-connected case } \la{DP}

To complete the proof of Conjecture 1, let us prove it for doubly-connected domains.

\begin{lemma} \la{l1}
Let $\Omega$ be a doubly-connected extremal domain with analytic boundary $\Gamma = \Gamma_1 \cup \Gamma_2.$ If $\varphi$ is the best analytic approximation to $\bar{z}$ in the supremum norm, and if $h$ is the conformal map from $\Omega$ onto an annulus $A:= \{w: R_2 \le |w| \le R_1\}$, then
    \be \la{cm}
        \varphi'(z) = C [(\log h(z))']^2,
    \ee
for some constant $C$.  In particular, $\varphi'$ is non-vanishing and $\sqrt{\varphi'}$ is single-valued in $\Omega.$
\end{lemma}

\begin{proof}  Since $\Omega$ is extremal, we have on $\Gamma$
    \be
        \bar{z} - i \lambda \dot{\bar{z}} = \varphi(z(s)),
    \ee
where $\lambda = 2 A/P,$ $s$ is the arc-length parameter, and $\dot{\bar{z}}= \frac{d \bar{z}}{ds}$. As before, differentiating with respect to arc-length gives
    \be
        \dot{\bar{z}} - i \lambda \ddot{\bar{z}} = \varphi'(z) \cdot \dot{z},
    \ee
and dividing by $\dot{\bar{z}}$ gives
    \be \la{real}
        1 -i \lambda \overline{\left( \frac{\ddot{z}}{\dot{z}}\right)} = \varphi'(z) \cdot (\dot{z})^2.
    \ee
Since $\dot{z}$ and $\ddot{z}$ are orthogonal, the left hand side of \eqref{real} is real-valued on $\Gamma,$ and therefore so is
$\varphi'(z) \cdot  (\dot{z})^2.$  Letting $h$ be the conformal map from $\Omega$ onto the annulus $A$ and $g = h^{-1}$  and writing $z = g(w)$ yields that
    \be \la{real2}
        \varphi'(z) \cdot  (\dot{z})^2 = \varphi'(g(w)) \cdot  \left( \frac{g'(w) dw}{|g'(w)| |dw|}\right)^2
    \ee
is real-valued on $\partial A$, and hence so is $ \varphi'(g(w)) \cdot  \left( g'(w) \frac{dw}{|dw|} \right)^2.$  Now notice that on $\partial A,$ $dw = \pm i w d \theta,$ and hence
$ \psi(w) = \varphi'(g(w)) \cdot  \left( g'(w) w \right)^2$ is a bounded analytic function in the annulus that is real-valued on $\partial A,$ and therefore is a constant $C.$ Rewriting
$\psi$ in terms of $z$ gives
    \be
        \varphi'(z) \cdot \left( \frac{h(z)}{h'(z)} \right)^2 = C,
    \ee
or $\varphi'(z) =  C \left( \frac{h'(z)}{h(z)} \right)^2 = C [(\log h(z))']^2,$ as desired.
\end{proof}

\begin{lemma} \la{l2}
The diffeomorphism  $\mu: \Gamma_2 \to \Gamma_1$, defined by
$$
\mu(z) = h^{-1}\left (\frac{R_1}{R_2}h(z) \right ),
$$
is a M\"obius transformation.
\end{lemma}
\begin{proof}
Clearly, $\mu$ is a diffeomorphism by composition law.
By definition, for any $z_2 \in \Gamma_2, z_1 := \mu(z_2) \in \Gamma_1$,
\be
\frac{h(z_1)}{h(z_2)} = \frac{R_1}{R_2}, \quad (h \circ \mu)(z_2) = \frac{R_1}{R_2}h(z_2),
\ee
so the chain rule and Lemma~\ref{l1} give
\be \la{rr}
h'(z_1)\cdot \mu'(z_2) = \frac{R_1}{R_2}h'(z_2) \Rightarrow \mu'(z_2) = \frac{h(z_1)}{h(z_2)}  \cdot \frac{h'(z_2)}{h'(z_1)} = \sqrt{\frac{\varphi'(z_2)}{\varphi'(z_1)}}.
\ee
Therefore,
\be \la{invar}
\frac{dz_1}{dz_2} = \sqrt{\frac{\varphi'(z_2)}{\varphi'(z_1)}} \Rightarrow \varphi'(z_2)dz_2^2 = \varphi'(z_1)dz_1^2.
\ee
Now note that $\varphi'$ is proportional to the Schwarzian of the ratio of any pair  $v_{i},v_{j}$ of solutions to \eqref{ODE} \cite[Ch. 6]{LG}:
$$
\varphi' = \frac{\lambda^2}{2} \mathcal{S}\left ( \frac{v_i}{v_j}
\right ) ,
$$
where
\be \la{Schwarz}
\mathcal{S}(f) \equiv \left (\log f' \right )'' - \frac{1}{2}[(\log f')']^2.
\ee
Since $\varphi'(z)$ is a Schwarzian, it transforms under composition with the map $\mu(z)$ as
\be
\varphi'(z_2)dz_2^2 = \varphi'(z_1)dz_1^2 + \mathcal{S}(\mu(z_1))dz_1^2,
\ee
where $\mathcal{S}(\mu(z))$ is the Schwarzian of the map $\mu(z)$. Thus, \ref{invar} gives $\mathcal{S}(\mu) \equiv 0$, so
$\mu(z)$ is a M\"obius transformation.

\end{proof}

\begin{lemma} \la{l3}
Suppose $f$ is a conformal map from an annulus $A$ to a doubly-connected domain $\Omega.$ If there exists a M\"{o}bius transformation $\mu$ and a constant $\lambda \in \CC$ such that
$(\mu \circ f)(z) = f(\lambda z),$ then either $f$ is a linear function or there exist constants $a, b, \in \CC$ such that $f(z) = a/z + b.$
\end{lemma}

\begin{proof}
Since $(\mu \circ f)(z) = f(\lambda z),$ writing $D_{\lambda}(z) = \lambda z$ and taking Schwarzian derivatives of both sides of the equation, we get
$\mathcal{S}(\mu \circ f) = \mathcal{S}(f \circ D_{\lambda}) .$  But since Schwarzians are invariant under post composition with M\"{o}bius transformations, and since
$\mathcal{S}(f \circ D_{\lambda})(z) = \lambda^2 (\mathcal{S}f)(\lambda z),$ we obtain that
$(\mathcal{S}f)(z) = \lambda^2 (\mathcal{S}f)(\lambda z),$ implying that $\mathcal{S}f$ is a homogeneous function of order $-2$.  Therefore
$(\mathcal{S}f)(z) = \frac{c}{z^2}$ for some constant $c$.  Now using the definition of the Schwarzian given in Equation \eqref{Schwarz} and setting $u = (\log f')'$, we arrive
again at the Ricatti equation
$ u' - \frac{1}{2}u^2 = \frac{c}{z^2}$ discussed earlier.  This is a first order ODE, and one can easily see that the general solution is $u(z) = \frac{c_0}{z}$, where $c_0$ is a constant.  Therefore $(\log f')(z) = c_0 \log z + c_1$, or $f'(z) = c_2 z^{c_0}$, and hence, since $f$ is analytic in $A$, $f(z) = c_z z^m + c_3$, for constants $c_2,c_3$ and $m \in \ZZ$.  But since $f$ is a conformal map from the annulus to a doubly-connected domain $\Omega$, $m$ can only equal $\pm 1$.
Therefore $f$ is either a linear function or $f(z) = a/z + b$ for constants $a, b \in \CC.$
\end{proof}

\begin{thm} \la{final}
Let $\Omega$ be a doubly-connected extremal domain with analytic boundary.  Then $\Omega$ is an annulus.
\end{thm}

\begin{proof}
By Lemmas \ref{l1} and \ref{l2}, the conformal map $h^{-1}$ from a given annulus $A$ to the region $\Omega$ satisfies the conditions of Lemma \ref{l3}.  Therefore, $h^{-1}$ is either linear or of the form $a/z + b$ for constants $a, b \in \CC,$ and hence $\Omega$ is the image of an annulus under either a linear map or a map of the form  $a/z + b.$ Therefore $\Omega$ must itself be an annulus.
\end{proof}

\section{Concluding Remarks}

Let us briefly outline several remaining open questions.

(I) The proof of Conjecture 1 hinges entirely on an a priori assumption that the extremal domain $\Omega$ is finitely connected and
has analytic boundary.  Yet, conditions (i) through (iv) of Theorem \ref{DK} make perfect sense if we only assume that $\Gamma = \p \Omega$ consists
of Jordan rectifiable curves.  (Of course, in that case, one requires that (ii) and the second equation in (iv) hold almost everywhere
on $\Gamma$.) It is rather natural to conjecture that (i) in Theorem \ref{DK} already enforces severe regularity assumptions on the free boundary
$\Gamma$ of $\Omega$.  Perhaps techniques from \cite{CKS} can be adjusted to make some headway on this question.  However, one must always be cautious, since
highly irregular non-Smirnov pseudo circles with rectifiable boundaries can easily arise in connection with problems similar to \eqref{fbp} (see the discussion in \cite{EKS,KSV}).

(II) The concept of analytic content has been extended to $\RR^n$ in \cite{GuKh} as the uniform distance from the identity vector field to divergence and curl free
vector fields (harmonic vector fields).  It was shown in \cite{GuKh} that an analogue of \eqref{geo} holds, namely
\be
\frac{n \cdot Volume{(\Omega)}}{P(\p \Omega)} \leq \lambda( \Omega) \leq c R_{vol},
\ee
but for some constant $c > 1$.  It would be interesting to know whether this constant can be replaced by $1$ for $n > 2$ - the proof in \cite{GuKh} cannot be tightened to obtain
$c = 1$; however, no example with $c > 1$ is known.  The authors of \cite{GuKh} proved that the analogue of Theorem \ref{DK} holds in $\RR^n$ and conjectured that the lower bound is
attained only for balls and spherical shells.  Furthermore, note that if the extremal domain is homeomorphic to a ball, it must be a ball of radius $\lambda$ (\cite[Corollary 3.3]{GuKh}).
Yet, without any constraints on the constants $c_k$ in (iv) of Theorem \ref{DK}, the problem of identifying the extremal domain remains wide open.
Finally, the question of the regularity requirement for the boundary of the extremal domain raised in (I) remains unknown in $\RR^n$ as well.

(III) Extending (ii) of Theorem \ref{DK} to the more general free boundary problem \eqref{fbp} with a meromorphic (instead of analytic) right hand side seems natural.
Virtually nothing is known except for rather limited results when either $p = 0$ or $t = 0$ (see \cite{KSV,EKS,BeKh}).

(IV) An intriguing consequence of Theorem~\ref{final}, when applied to the problem described in \S~\ref{quantum} and in \cite{RT2014}, is that in a CFT with multiple insertion points (one for each $\Omega_k$), either $n=2$ and the ``Planck constant" $\lambda = R_1-R_2 > 0$ can be taken arbitrarily small (as expected), or $n \ge 3$ and $\lambda$ is bounded from below, which would present an obstacle problem for deformation quantization in two dimensions.


 \bibliographystyle{amsplain}

\end{document}